\documentclass[12pt]{article}
\usepackage{amssymb,amsfonts,amsmath,amsthm,latexsym, epsfig,mathrsfs,colordvi}
\usepackage{graphicx}
\newtheorem{thm}{Theorem}

\newtheorem{coro}[thm]{Corollary}
\newtheorem{prop}[thm]{Proposition}

\newcommand{\be}{\begin{enumerate}}
\newcommand{\ee}{\end{enumerate}}

\newcommand{\beq}{\begin{equation}}
\newcommand{\eeq}{\end{equation}}
\newcommand{\beas}{\begin{eqnarray*}}
\newcommand{\eeas}{\end{eqnarray*}}
\newcommand{\bea}{\begin{eqnarray}}
\newcommand{\eea}{\end{eqnarray}}

\newcommand{\st}{\,:\,}

\theoremstyle{definition}
\newtheorem{defn}[thm]{Definition}

\newtheorem{conj}[thm]{Conjecture}

\newcommand{\ffp}{\mathbb{F}_p}
\newcommand{\zz}{\mathbb{Z}}
\newcommand{\rr}{\mathbb{R}}
\newcommand{\pp}{\mathbb{P}}

\newcommand{\sn}{\mathfrak{S}_n}

\newcommand{\ca}{\mathcal{A}}
\newcommand{\calc}{\mathcal{C}}
\newcommand{\ci}{\mathcal{I}}
\newcommand{\vo}{\mathrm{vo}}
\newcommand{\cp}{\mathcal{P}}
\newcommand{\aff}{\mathrm{aff}}
\newcommand{\vis}{\mathrm{vis}}
\newcommand{\orp}{\mathcal{O}(P)}
\newcommand{\sco}{\mathrm{sc}}

\makeatletter \@addtoreset{equation}{section}
\@addtoreset{theo}{}\makeatother

\begin{document}

\title{Valid Orderings of Real Hyperplane Arrangements}

\author{Richard P.
  Stanley\footnote{ Department of Mathematics, Massachusetts Institute
    of Mathematics, Cambridge, MA 02139, USA. Email:
    rstan@math.mit.edu. Based
    upon work supported by the National 
    Science Foundation under Grant No.~DMS-1068625.}}

\date{June 7, 2013}

\maketitle

\noindent{\bf Abstract.} Given a real finite hyperplane arrangement
$\ca$ and a point $p$ not on any of the hyperplanes, we define an
arrangement $\vo(\ca,p)$, called the \emph{valid order arrangement},
whose regions correspond to the different orders in which a line
through $p$ can cross the hyperplanes in $\ca$. If $\ca$ is the set of
affine spans of the facets of a convex polytope $\cp$ and $p$ lies in
the interior of $\cp$, then the valid orderings with respect to $p$
are just the line shellings of $\cp$ where the shelling line contains
$p$. When $p$ is sufficiently generic, the intersection lattice of
$\vo(\ca,p)$ is the \emph{Dilworth truncation} of the semicone of
$\ca$. Various applications and examples are given. For instance, we
determine the maximum number of line shellings of a $d$-polytope with
$m$ facets when the shelling line contains a fixed point $p$. If $\cp$
is the order polytope of a poset, then the sets of facets visible from
a point involve a generalization of chromatic polynomials related to
list colorings.



\noindent

\section{Introduction}
Let $\ca$ be a (finite) real hyperplane arrangement, i.e., a finite
set of affine hyperplanes in some $d$-dimensional real affine space
$V\cong \rr^d$. Since we consider only hyperplane arrangements in this
paper, we call $\ca$ simply a real arrangement, always assumed to be
finite. Basic information on arrangements may be found in Orlik and
Terao \cite{or-te} and Stanley \cite{rs:hyp}.

The main question that will concern us is the following. Let $L$ be a
directed line in $V$. If $L$ is sufficiently generic then it will
cross the hyperplanes $H\in\ca$ in a certain order. 
What can we say about the possible orders of the hyperplanes?
We can say more when we fix a point $p\in V$ not lying on
any of the hyperplanes in $\ca$ and assume that $L$ passes through
$p$. The different orders then correspond in a simple way to regions
of another arrangement, which we call the \emph{valid order
arrangement} $\vo(\ca,p)$.

A special situation occurs when $\ca$ consists of the affine spans of
the facets of a $d$-dimensional convex polytope $\cp$ in $\rr^d$. We
then call $\ca$ the \emph{visibility arrangement} vis$(\cp)$ of $\cp$,
since its regions correspond to sets of facets of $\cp$ visible from
some point. If $p$ lies in the interior of $\cp$, then the regions of
the valid order arrangement $\vo(\ca)$ correspond to the line
shellings of $\cp$, where the line defining the shelling (which we
call the \emph{shelling line}) passes through $p$. In this case we
call $\vo(\ca,p)$ the \emph{line shelling arrangement} of $\cp$ (with
respect to $p$).

We will discuss a number of results concerning visibility and valid
order arrangements. Most notably, when $p$ is sufficiently generic,
then the matroid corresponding to the semicone (defined below) of
$\vo(\ca,p)$ is the Dilworth 
truncation of the matroid corresponding to $\ca$. This observation
enables us (Theorem~\ref{thm:max}) to answer the following question:
given $n\geq d+1$, what is the most number of line shellings that a
convex $d$-polytope with $n$ facets can have, where the shelling line
passses through a fixed point $p$? Another result
(Theorem~\ref{thm:vop}) is a connection between the visibility
arrangement of the order polytope of a poset and a generalization of
chromatic polynomials.

\section{The valid order arrangment}
Let $\ca$ be a hyperplane arrangement in a real affine space $V$, and
let $p$ be a point in $V$ not lying on any hyperplane $H\in\ca$. 

\begin{defn}
The \emph{valid order arrangement} $\vo(\ca,p)$ consists of all
hyperplanes of the following two types:
 \begin{itemize} 
   \item The affine span of $p$ and $H\cap H'$, where $H$ and
     $H'$ are two non-parallel hyperplanes in $\ca$. We denote this
     affine span as aff$(p,H\cap H')$.
   \item The hyperplane through $p$ parallel to two parallel
     hyperplanes $H,H'\in \ca$, denoted par$(p,H)$ or par$(p,H')$. 
  \end{itemize}
\end{defn}

Note that $\vo(\ca,p)$ is a \emph{central} arrangement, i.e., all the
hyperplanes in $\vo(\ca,p)$ intersect, since every hyperplane in
$\vo(\ca,p)$ contains $p$.

Consider a directed line $L$ through $p$ that is not parallel to any
hyperplane $H\in\ca$ and that does not intersect two distinct
hyperplanes of $\ca$ in the same point. Thus $L$ intersects the
hyperplanes in $\ca$ in some order $H_1,H_2,\dots,H_m$ as we come in
from $\infty$ along $L$ in the direction of $L$. We call the sequence
$H_1,\dots,H_m$ a \emph{valid ordering} of $\ca$ with respect to
$p$. Note that if we reverse the direction of $L$, then we get a new
valid ordering $H_m,\dots, H_1$.

Suppose that $H$ and $H'$ are two non-parallel hyperplanes of
$\ca$. The question of whether $L$ intersects $H$ before $H'$ depends
on which side of the hyperplane aff$(p,H\cap H')$ a point $q$ lies,
where $q$ is a point of $L$ near $p$ in the positive direction (the
direction of $L$) from $p$. Similarly, if $H$ and $H'$ are parallel
hyperplanes of $\ca$, then either $q$ lies on the same side of both
(i.e., not between them), in which case the order in which $L$
intersects $H$ and $H'$ is independent of $L$, or else $q$ lies
between $H$ and $H'$, in which case the order in which $L$ intersects
$H$ and $H'$ depends on which side of the hyperplane par$(p,H)$ the
point $q$ lies. It follows that the valid ordering corresponding to
$L$ is determined by which region of $\vo(\ca,p)$ the point $q$
lies. In particular, we have the following result.

\begin{prop}
The number of valid orderings of $\ca$ with respect to $p$ is equal to
the number $r(\vo(\ca,p))$ of regions of the valid order arrangement
$\vo(\ca,p)$
\end{prop}

We now wish to explain the connection between the valid order
arrangement and a matroidal construction known as ``Dilworth
truncation.'' Recall that a \emph{matroid} on a set $E$ may be defined
as a collection $\ci$ of subsets of $E$, called \emph{independent
  sets}, satisfying the following condition: for any subset
$F\subseteq E$, the maximal (under inclusion) sets in $\ci$ that
are contained in $F$ all have the same number of elements. The
protypical example of a matroid consists of a finite subset $E$ of a
vector space, where a set $F\subseteq E$ is independent if it is
linearly independent. For further information on matroid theory, see
for instance \cite{oxley}\cite{welsh}\cite{white1}.

We first define a matroid $M_\ca$ associated with an
arrangement. Given a real arrangement $\ca$ in a vector space $V$
which we identify with $\rr^d$, let $H$ be a hyperplane in $\ca$
defined by the equation $x\cdot \alpha=c$, where $0\neq
\alpha\in\rr^d$ and $c\in\rr$. Associate with $H$ the vector
$v_H=(\alpha,-c)\in\rr^{n+1}$. Let $M_\ca$ be the matroid
corresponding to the set $E_\ca=\{v_H\st H\in\ca\}$.  That is, the
points of $M_\ca$ are the vectors in $E_\ca$, with independence in
$M_\ca$ given by linearly independence of vectors. Note the $M_\ca$ is
a \emph{linear arrangement}, that is, all its hyperplanes pass through
the origin.

\textsc{Note.} Denote the coordinates in $\rr^{n+1}$ by
$x_1,\dots,x_n,y$. Preserving the notation from above, let $\sco(\ca)$
denote the set of all hyperplanes $\alpha\cdot x=cy$ in
$\rr^{n+1}$. We call $\sco(\ca)$ the \emph{semicone} of $\ca$. If we
add the additional hyperplane $y=0$, then we obtain the \emph{cone}
$c(\ca)$, as defined e.g.\ in \cite[{\S}1.1]{rs:hyp}. Note that
$\sco(\ca)$ is a linear arrangement satisfying $M_\ca\cong
M_{\sco(\ca)}$.

Now let $M$ be a matroid on a set $E$, and let $L=L_M$ denote the
lattice of flats of $M$. If we remove the top $k$ levels from $L$
below the maximum element $\hat{1}$, then we obtain the $k$th
\emph{truncation} $T^kL$ of $L$. It is easy to see that $T^kL$ is a
geometric lattice and hence the lattice of flats of a matroid. What
if, however, we remove the \emph{bottom} $k$ levels from $L$ above the
minimum element $\hat{0}$? In general, we do not obtain a geometric
lattice. We would like to ``fill in'' this lower truncation as
generically as possible to obtain a geometric lattice, without adding
any new atoms (elements of rank $k+1$ of $L$), and without increasing
the rank. This rather vague description was formalized by Dilworth
\cite{dil}. Three other references are Brylawski \cite{bry}\cite{bry2}
and Mason \cite{mason}. We will give the definition at the level of
matroids. Define the $k$th \emph{Dilworth truncation} $D_kM$ to be the
matroid on the set $\binom{E}{k+1}$ of $(k+1)$-element subsets of $E$,
with independent sets
    $$ \mathcal{I} =\left\{ I\subseteq \binom{E}{k+1}\st
   \mathrm{rank}_M\left(\bigcup_{p\in I'}p\right) \geq \#I'+k,\ \ 
   \forall\, \emptyset\neq I'\subseteq I\right\}. $$
Thus the flats of rank one of $D_kM$ are just the flats of rank $k+1$
of $M$. In particular, the flats of $D_1M$ are the lines (flats of
rank two) of $M$. We carry over the notation $D_k$ to geometric
lattices. In other words, if $L$ is a
geometric lattice, so $L=L_M$ for some matroid $M$, then we define
$D_kL = L_{D_kM}$. 

\textsc{Note.} Various other notations are used for $D_k$, including
$D_{k+1}$ and $T_{k+1}$.

In general, $D_1L$ seems to be an intractable object. For the boolean
algebra $B_m$ we have \cite[Thm.~3.2]{dil}\cite[p.~163]{mason}
  \beq D_1 B_m\cong \Pi_m, \label{eq:d1bm} \eeq
the lattice of partitions of an $m$-set (or the intersection lattice
of the braid arrangement $\mathcal{B}_m$), but for more complicated
geometric lattices $L$ it is difficult to describe $D_1L$ in a
reasonable way. If $L$ has rank two then clearly $D_1L$ consists of
just two points $\hat{0}$ and $\hat{1}$. If $L$ has
rank three then when we remove the atoms from $L$ we still have a
geometric lattice, so $D_1L$ consists just of $L$ with the atoms
removed. When $L$ has rank four, to obtain $D_1L$ first remove the atoms
from $L$ to obtain a lattice $L'$ of rank three. For any two atoms
$s,t$ of $L'$ whose join in $L'$ is the top element $\hat{1}$ of $L'$,
adjoin a new element $x_{st}$ covering $s$ and $t$ and covered by
$\hat{1}$. The resulting poset is $D_1L$. This construction allows us
to give a formula for the characteristic polynomial (e.g.,
\cite[{\S}1.3]{rs:hyp}\cite[{\S}3.11.2]{ec1}) of $D_1L$ when
rank$(L)=4$. Let $\rho_2$ be the number of elements of $L$ of rank
two, let $L_3$ be the set of elements of $L$ of rank three, and let
$c(t)$ be the number of elements $u$ covering $t\in L$, i.e., $u>t$,
and no element $v$ satisfies $u>v>t$). Then
  \beas \chi_{D_1L}(q) & = & q^3-\rho_2\,q^2+\left[
  \binom{\rho_2}{2} - \sum_{t\in L_3}\binom{c(t)-1}{2}\right]q\\
   & & \quad +\sum_{t\in L_3}\binom{c(t)-1}{2}-\binom{\rho_2-1}{2}. 
  \eeas
When rank$(L)=5$ the situation becomes much more complicated.

We now come to our main result on the valid order arrangement.

\begin{thm} \label{thm:dt}
Let $\ca$ be an arrangement in the real vector space $V$, and let $p$
be a generic point of $V$. Then $L_{\vo(\ca,p)}\cong L_{D_1(\ca)}$.
\end{thm}

\begin{proof}
Brylawski \cite[p.~62]{bry}\cite[p.~197]{bry2} and Mason
\cite[pp.~161--162]{mason} note that the Dilworth truncation of a
geometry (simple matroid) $M$ embedded in a vector space $V$ of the
same dimension (over a sufficiently large field if the field
characteristic in nonzero) is obtained as the set of intersections of
the lines of $M$ with a generic hyperplane in $V$. This is precisely
dual to the statement of our theorem.
\end{proof}

As an example illustrating Theorem~\ref{thm:dt},
Figure~\ref{fig:voex}(a) shows an arrangement $\ca$ of four
hyperplanes (solid lines) in $\rr^2$ and a nongeneric point $p$. The
dashed lines are the hyperplanes in $\vo(\ca)$. The point $p$ is not
generic since the same hyperplane of $\vo(\ca)$ passes through the two
intersections marked $a$ and $b$. The arrangement $\vo(\ca)$ has ten
regions, so there are ten valid orderings of the four hyperplanes of
$\ca$ with respect to $p$. Figure~\ref{fig:voex}(b) shows the same
situation with a generic point $p$. There are now twelve valid
orderings with respect to $p$. In this case the lattice $L_\ca$ is an
(upper) truncated boolean algebra $T^1B_4$, with four atoms and six
elements of rank two. Since rank$(L_\ca)=3$ the Dilworth truncation
$D_1(L_\ca)$ is obtained simply by removing the atoms from $L_\ca$.

\begin{figure}
\centering
\centerline{\includegraphics[width=12cm]{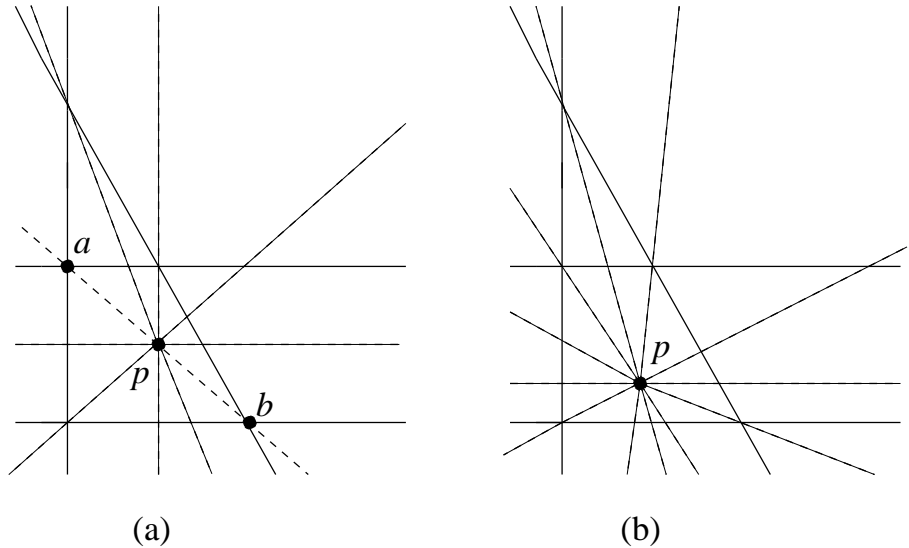}}
\caption{Two valid order arrangements} 
\label{fig:voex}
\end{figure}

\section{Examples}
As mentioned in the introduction, a special situation of interest
occurs when $\ca$ consists of the affine spans aff$(F)$ of the facets
$F$ of a $d$-dimensional convex polytope $\cp$ in $\rr^d$, is which
case we call $\ca$ the \emph{visibility arrangement} vis$(\cp)$ of
$\cp$. The regions of vis$(\cp)$ correspond to the sets of facets that
are visible (on the outside) from some point in $\rr^d$. In
particular, the interior of $\cp$ is a region from which \emph{no}
facets are visible.
Let
$v(\cp)=r(\vis(\cp))$, the number of regions of $\vis(\cp)$ or
visibility sets of facets of $\cp$.  If $p$ is a point inside $\cp$,
then the valid orderings $(\aff(F_1),\dots,\aff(F_r))$ with respect to
$p$ correspond to the line shellings $(F_1,\dots,F_r)$ where the
shelling line passes through $p$. For basic information on line
shellings, see Ziegler \cite[Lecture~8]{ziegler}.

For an arrangement $\ca$ in $\rr^d$, let $\chi_\ca(q)$ denote the
\emph{characteristic polynomial} of $\ca$ (e.g.,
\cite[{\S}1.3]{rs:hyp}\cite[{\S}3.11.2]{ec1}). A well-known theorem of
Zaslavsky \cite[Thm.~2.5]{rs:hyp}\cite[Thm.~3.11.7]{ec1} states that
the number $r(\ca)$ of regions of $\ca$ is given by 
   \beq r(\ca) = (-1)^d\chi_\ca(-1). \label{eq:zas} \eeq 
Suppose that $\ca$ is defined over $\zz$, that is, the equations
defining the hyperplanes in $\ca$ have integer coefficients. By taking
these coefficients modulo a prime $p$, we get an arrangement $\ca_p$
defined over the finite field $\ffp$. It is also well-known
\cite[Thm.~5.15]{rs:hyp}\cite[Thm.~3.11.10]{ec1} that for $p$
sufficiently large, 
   \beq \chi_\ca(p) = \#\left(\ffp^d-\bigcup_{H\in\ca_p}H\right).
    \label{eq:ff} \eeq
This result will be a useful tool below in computing some
characteristic polynomials. 

We now discuss two examples, the $n$-cube and the order polytope of a
finite poset.  Let $\calc_n$ denote the standard $n$-dimensional cube,
given by the inequalities $0\leq x_i\leq 1$, for $1\leq i\leq n$. It
is easy to see, e.g., by equation~\eqref{eq:ff},  that the visibility
arrangement vis$(\calc_n)$ satisfies
  $$ \chi_{\vis(\calc_n)}(q) = (q-2)^n. $$
In particular, $r(\vis(\calc_n)) = 3^n$. Drawing a picture for $n=2$
will make it geometrically clear why $\vis(\calc_n)$ has $3^n$
regions. In fact, the facets of $\calc_n$ come in $n$ antipodal pairs
$F$ and $\bar{F}$. The sets of facets visible from some point are
obtained by choosing for each pair $F,\bar{F}$ either $F$, $\bar{F}$,
or neither. There are three choices for each pair, so $3^n$ visibility
sets in all.

More interesting are the line shellings of cubes.
We summarize some information in the following result.

\begin{thm} \label{thm:cube}
  \be\item[(a)] Let $p=\left(\frac 12,\frac 12,\dots,
\frac 12\right)$, the center of the cube $\calc_n$. Then
  $$ \chi_{\vo(\vis(\calc_n),p)}(q) = (q-1)(q-3)\cdots (q-(2n-1)), $$ 
so the number of line shellings with respect to $p$ is $2^nn!$.
  \item[(b)] The \emph{total} number of line shellings of $\calc_n$ is
    $2^nn!^2$. 
  \item[(c)] Let $f(n)$ denote the total number of shellings of
    $\calc_n$. Then
   \beq \sum_{n\geq 1}f(n)\frac{x^n}{n!} =1-\frac{1}{\sum_{n\geq 0}
      (2n)!\frac{x^n}{n!}}. \label{eq:cubesh} \eeq
   \item[(d)] Every shelling of $\calc_n$ can be realized as a
     corresponding line
     shelling of a polytope combinatorially equivalent to $\calc_n$. 
 \ee
\end{thm}

\begin{proof}
 \be\item[(a)] The hyperplanes of $\vo(\vis(\calc_n))$ are given by
 $x_i=0$, $1\leq i\leq n$, and $x_i\pm x_j=0$, $1\leq i<j\leq 1$.The
 characteristic polynomial can now easily be computed from
 \eqref{eq:ff}. Alternatively, $\vo(\vis(\calc_n))$ is the
 \emph{Coxeter arrangement} of type $B_n$, whose characteristic
 polynomial is well-known
 \cite[p.~451]{rs:hyp}\cite[Exer.~3.115(d)]{ec1}. 
  \item[(b)] If we stand at a generic point far away from $\calc_n$ we
    will see $n$ facets of $\calc_n$, all with a common vertex $v$. By
    symmetry, there are $2^n$ choices for $v$, and then $n!$ orderings
    of the $n$ facets containing $v$ that can begin a line
    shelling $\sigma$. Hence it remains to prove that the remaining
    $n$ facets can come in any order in $\sigma$.
 
    Let the parametric equation of the line $L$ defining the shelling
    be $(a_1,a_2,\dots,a_n)+t(\alpha_1,\alpha_2,\dots,\alpha_n)$,
    where $t\in\rr$. Making a small perturbation if necessary, we may
    assume that each $\alpha_i\neq 0$. We may also assume by
    symmetry that the facet $F_i$ of the shelling, for
    $1\leq i\leq n$, has the equation $x_i=0$. The line $L$ intersects
    the hyperplane $x_i=0$ when $t=-a_i/\alpha_i$, so
      $$ \frac{a_1}{\alpha_1}>\frac{a_2}{\alpha_2}>\cdots>
        \frac{a_n}{\alpha_n}. $$
    The line $L$ intersects the hyperplane $x_i=1$ when
    $t=(1-a_i)/\alpha_i$. Write 
     $$ \frac{1-a_i}{\alpha_i} = \frac{1}{\alpha_i}+b_i, $$ 
    so $b_1<b_2<\cdots<b_n$. Thus we can first choose
    $b_1<b_2<\cdots<b_n$. Then choose
    $\alpha_1,\alpha_2,\dots,\alpha_n$ so that the numbers
    $\frac{1}{\alpha_i}+b_i$ come in any desired order. This then
    determines $a_1,\dots, a_n$ uniquely, completing the proof.
  \item[(c)] This result is stated without proof in
    \cite[Exer.~1.131]{ec1}. To prove it, note that $F_1,F_2,\dots,
    F_{2n}$ is a shelling if and only if for no $1\leq j<n$ is it true
    that $\{F_1,F_2,\dots, F_{2j}\}$ consists of $j$ pairs of
    antipodal facets. There follows the recurrence
   $$ (2n)! =\sum_{j=0}^n f(j)\binom nj(2n-2j)!, $$
 from which equation~\eqref{eq:cubesh} is immediate. 
  \item[(d)] See M. L. Develin \cite[Cor.~2.12]{dev}.
 \ee
\end{proof}

Conspicuously absent from Theorem~\ref{thm:cube} is the characteristic
polynomial or number of regions of the line shelling arrangement
vo$(\vis(\calc_n),p)$ when $p$ is \emph{generic}, the situation of
Theorem~\ref{thm:dt}. Suppose for instance that $n=3$. Let $\ca(p)=
\vo(\vis(\calc_3),p)$. When $p=\left(
\frac 12,\frac 12,\frac 12\right)$, then by Theorem~\ref{thm:cube}(a)
we have 
  $$ \chi_{\ca(p)}(q) = (q-1)(q-3)(q-5),\ \ r(\ca)=48. $$
For $p=\left(\frac 12,\frac 12,\frac 14\right)$ we have
  $$ \chi_{\ca(p)}(q) = (q-1)(q-5)(q-7),\ \ r(\ca)=96. $$
For generic $p$ we have
  $$ \chi_{\ca(p)}(q) = (q-1)(q^2-14q+53),\ \ r(\ca)=136=2^3\cdot 17. 
  $$
The total number of line shellings of $\calc_3$ is 288, and the total
number of shellings in 480. While the Dilworth truncation
$D_1(\vis(\calc_n))$ seems quite complicated, it might not be hopeless
to compute its characteristic polynomial or number of regions. We
leave this as an open problem.

We next consider the order polytope $\orp$ of a finite poset $P$,
first defined explicitly in \cite{op}. By definition, $\orp$ is the
set of all order-preserving maps $\tau\colon P\to [0,1]$ and is hence
a convex polytope in the space $\rr^P$ of all maps $P\to \rr$. Our
main result will be a connection between the number of regions of
$\vis(\orp)$, i.e., the number of visibility sets of facets of $\orp$,
and a certain generalization of the chromatic polynomial of a graph.

Let $G$ be a finite simple (i.e., no loops or multiple edges) graph
with vertex set $V$. Recall that a \emph{proper coloring} of $G$ with
colors from the set $\pp$ of positive integers is a map $f\colon
V\to \pp$ such that if $u$ and $v$ are adjacent in $G$ then
$f(u)\neq f(v)$. The \emph{chromatic polynomial} $\chi_G(q)$
is defined when $q\in\pp$ to be the number of proper colorings
$f\colon V\to \{1,2,\dots,q\}$. It is a standard result that
$\chi_G(q)$ is a polynomial in $q$. Moreover, if
$V=\{v_1,\dots,v_p\}$, then define the \emph{graphical arrangement}
$\ca_G$ to be the arrangement in $\rr^p$ with hyperplanes $x_i=x_j$,
where $v_i$ and $v_j$ are adjacent vertices of $G$. Then
$\chi_{\ca_G}(q) = \ca_G(q)$
\cite[Thm.~2.7]{rs:hyp}\cite[Exer.~3.108]{ec1}. 

We will generalize the definition of $\chi_G(q)$ by imposing finitely
many disallowed colors at each vertex. More precisely, let $2^\pp$
denote the set of all subsets of $\pp$, and let $\psi\colon V\to
2^\pp$ satisfy $\#\psi(v)<\infty$ for all $v\in V$. For $q\in\pp$,
define $\chi_{G,\psi}(q)$ to be the number of proper colorings
$f\colon V\to \{1,2,\dots,q\}$ such that $f(v)\not\in\psi(v)$ for all
$v\in V$. Thus for each vertex $v$, there is a finite set $\psi(v)$
of ``disallowed colors.'' We call such a coloring a
\emph{$\psi$-coloring}. The idea of permitting only certain colors
of each vertex in a proper coloring of $G$ has received much attention
in the context of \emph{list colorings} \cite{list}, but the function
$\chi_{G,\psi}(q)$ seems to be new.

It is easy to see that $\chi_{G,\psi}(q)$ is a monic polynomial in
$q$ of degree $p$ with integer coefficients. We call it the
\emph{$\psi$-chromatic polynomial} of $G$. Define the
\emph{$\psi$-graphical arrangement} $\ca_{G,\psi}$ to be the
arrangement in $\rr^p$ with hyperplanes $x_i=x_j$ whenever $v_i$ and
$v_j$ are adjacent in $V$, together with $x_i=\alpha_j$ if
$\alpha_j\in\psi(v_i)$.

\begin{thm} \label{thm:sigarr}
We have $\chi_{\ca_{G,\psi}}(q)=\chi_{G,\psi}(q)$, that is, the
$\psi$-chromatic polynomial of $G$ coincides with the characteristic
polynomial of the $\psi$-graphical arrangement $\ca_{G,\psi}$. 
\end{thm}

\begin{proof}
The proof is an immediate consequence of equation~\eqref{eq:ff}. 
\end{proof}

Because $\chi_{G,\psi}$ is the characteristic polynomial of a
hyperplane arrangement, it satisfies all the properties of such
polynomials. For instance, there is a deletion-contraction recurrence,
a broken circuit theorem, an extension to the Tutte polynomial,
etc. We now give the connection between $\vis(\orp)$ and
$\psi$-graphical arrangements.

\begin{thm} \label{thm:vop}
Let $P$ be a finite poset, and let $H$ denote the Hasse diagram of
$P$, considered as a graph with vertex set $V$. Define $\psi\colon
V\to \pp$ by
    $$ \psi(v) = \left\{ \begin{array}{rl} \{1,2\}, &
    \mathrm{if}\ v\ \mathrm{is\ an\ isolated\ point}\\
    \{1\}, &
    \mathrm{if}\ v\ \mathrm{is\ minimal\ but\ not\ maximal}\\ 
    \{2\}, & \mathrm{if}\ v\ \mathrm{is\ maximal\ but\ not\ minimal}\\
     \emptyset, & \mathrm{otherwise}. \end{array} \right. $$
Then $\vis(\orp)+(1,1,\dots,1)=\ca_{H,\psi}$, where
$\vis(\orp)+(1,1,\dots,1)$ denotes the translation of  $\vis(\orp)$ by
the vector $(1,1,\dots,1)$.
\end{thm}

\begin{proof}
The result is an immediate consequence of the relevant
definitions. Namely, if $V=\{v_1,\dots,v_p\}$ then the facets of
$\orp$ are given by
  $$ \begin{array}{llrl} x_i & = & x_j, &
    \mathrm{if}\ v_j\ \mathrm{covers}\ v_i\  
     \mathrm{in}\ P\\
        x_i & = & 0, &
        \mathrm{if}\ x_i\ \mathrm{is\ a\ minimal\ element\ of}\ P\\ 
        x_i & = & 1, &
        \mathrm{if}\ x_i\ \mathrm{is\ a\ maximal\ element\ of}\ P,
   \end{array} $$
and the proof follows.  
\end{proof}

\textsc{Note.} We could have avoided the translation by
$(1,1,\dots,1)$ by allowing 0 to be a color, but it is more natural in
many situations to let the set of colors be $\pp$.

A curious result arises when $P$ is graded of rank one, i.e., every
maximal chain of $P$ has two elements. For $W\subseteq V$, let $H_W$
be the restriction of $H$ to $W$, or in other words, the induced
subgraph on the vertex set $W$.

\begin{thm} \label{thm:chim3}
Suppose that $P$ is graded of rank one. Then 
  \bea \chi_{\vis(\orp)}(q) & = & \sum_{W\subseteq V}
  \chi_{H_W}(q-2) \label{eq:chirank1}\\[.5em]  
   v(\mathcal{O}(P)) & = & (-1)^{\#P}\sum_{W\subseteq V} \chi_{H_W}(-3).
   \label{eq:rrank1}
 \eea
\end{thm}

\begin{proof}
Let $q\geq 2$. Choose a subset $W\subseteq V$. Color each minimal
element of $P$ not in $W$ 
with the color 2, and color each maximal element of $P$ not in $W$
with the color 1. Color the remaining elements with the colors
$\{3,4,\dots,q\}$ in $\chi_{H_W}(q-2)$ ways. This produces each
  $\psi$-coloring of $H$, so the proof of
  equation~\eqref{eq:chirank1} follows. To obtain
  equation~\eqref{eq:rrank1}, put $q=-1$ in \eqref{eq:chirank1}.
\end{proof}

As an example, let $P_{mn}$ denote the poset of rank one with $m$
minimal elements, $n$ maximal elements, and $u<v$ for every minimal
element $u$ and maximal element $v$. Hence $H$ is the complete
bipartite graph $K_{mn}$. It is known \cite[Exer.~5.6]{ec2} that
 $$ \sum_{m\geq 0}\sum_{n\geq 0}\chi_{K_{mn}}(q)\frac{x^m}{m!}
   \frac{x^n}{n!} = (e^x+e^y-1)^q. $$
By simple properties of exponential generating functions we get
  $$ \sum_{m\geq 0}\sum_{n\geq 0}\chi_{\vis(\mathcal{O}(P_{mn}))}(q)
   \frac{x^m}{m!}\frac{x^n}{n!} =e^{x+y} (e^x+e^y-1)^{q-2} $$
and
  \beas \sum_{m\geq 0}\sum_{n\geq 0}v(\mathcal{O}(P_{mn}))
   \frac{x^m}{m!}\frac{x^n}{n!} & = & e^{-x-y}
   (e^{-x}+e^{-y}-1)^{-3}\\
     & = & 1+2(x+y)+7xy+4\frac{x^2+y^2}{2!}+
        23\frac{x^2y+xy^2}{2!}\\ & & \ +115\frac{x^2y^2}{2!^2}
        \ +8\frac{x^3+y^3}{3!}+73\frac{x^3y+xy^3}{3!}\\ & & \
    +533\frac{x^3y^2+x^2y^3}{2!\,3!}+3451\frac{x^3y^3}{3!^2}+\cdots.
   \eeas
For instance, the order polytope of $P_{22}$ has eight facets and 115
visibility sets of facets.

We now pose the question of extending some results on graphical
arrangements to $\psi$-graphical arrangements. An arrangement
$\ca$ is \emph{supersolvable} if the intersection lattice $L_{c(\ca)}$
of the cone $c(\ca)$ contains a maximal chain of 
modular elements. See for instance \cite{rs:hyp} for further
details. If $\ca$ is supersolvable, then every zero of $\chi_\ca(q)$
is a nonnegative integer. A graphical arrangement $\ca_G$ is
supersolvable if and only if $G$ is a \emph{chordal} graph (also
called a \emph{triangulated} graph or \emph{rigid circuit} graph)
\cite[Cor.~4.10]{rs:hyp}. It is natural to ask for an extension of
this result to $\psi$-graphical arrangements. The proof of the
following result is straightforward and will be omitted.

\begin{thm} \label{thm:psigss}
Let $(G,\psi)$ be as above. Suppose that we can order the vertices
of $G$ as $v_1,\dots,v_p$ such that:
 \begin{itemize}\item $v_{i+1}$ connects to previous vertices
     along a  clique (so $G$ is chordal).
  \item If $i<j$ and $v_i$ is adjacent to $v_j$, then
    $\psi(v_j)\subseteq \psi(v_i)$.
 \end{itemize}
Then $\ca_{G,\psi}$ is supersolvable.
\end{thm} 

\begin{conj} \label{conj:ss}
The converse to Theorem~\ref{thm:psigss} holds, that is, if
$\ca_{G,\psi}$ is supersolvable then $(G,\psi)$ satisfies the
conditions of Theorem~\ref{thm:psigss}.
\end{conj}

We suspect that Conjecture~\ref{conj:ss} will not be so difficult to
prove. There are numerous characterizations of chordal graphs
\cite{chordal}. If Conjecture~\ref{conj:ss} is 
  true, then it would be interesting to investigate which of these
  characterizations have analogues for the pairs $(G,\psi)$ satisfying
  the conditions of Theorem~\ref{thm:psigss}.

A profound generalization of supersolvable arrangements is due to
H. Terao (e.g.\ \cite[Ch.~4]{or-te}\cite[Thm.~4.14]{rs:hyp}), called
\emph{free} arrangements. Freeness was defined originally for central
arrangements, but we can define a noncentral arrangement $\ca$ to be
free if the cone $c(\ca)$ is free. The ``factorization theorem'' of
Terao asserts that if $\ca$ is free then the zeros of $\chi_\ca(q)$
are nonnegative integers (with an algebraic interpretation). Every
supersolvable arrangement is free, and every free graphical
arrangement is supersolvable. This leads to a second conjecture, which
again may not be difficult to prove.

\begin{conj}
If $\ca_{G,\psi}$ is a free $\psi$-graphical arrangement, then
$\ca_{G,\psi}$ is supersolvable.
\end{conj}

\section{Applications}
One immediate application of Theorem~\ref{thm:dt} follows from the
matroidal definition of Dilworth truncation.

\begin{coro}
The characteristic polynomial $\chi_{\vo(\ca,p)}(q)$, where $p$ is
generic, is a matroidal invariant, that is, it depends only
on $L_\ca$. In particular, the number $v(\ca,p)$ of valid orderings
with respect to a generic point $p$ is a matroidal invariant and hence
is independent of the region in which $p$ lies.
\end{coro}

\begin{proof}
The Dilworth truncation $D_kL$ of a geometric lattice $L$ is defined
as $L_{D_kM}$, where $M$ is the matroid assocated to $L$. The proof
that $L_{\vo(\ca,p)}(q)$ is a matroidal invariant
follows from Theorem~\ref{thm:dt}. The statement for $v(\ca,p)$ then
follows from Zaslavsky's theorem~\eqref{eq:zas}. 
\end{proof}

For our second application, let $c(n,k)$ denote the signless Stirling
number of the first kind, i.e., the number of permutations $w\in\sn$
with $k$ cycles. 

\begin{thm} \label{thm:max}
Let $\ca$ be an arrangement in $\rr^d$ with $m$ hyperplanes, and let
$p$ be a point in $\rr^d$ not lying on any $H\in\ca$. Then
  $$ v(\ca,p)\leq 2(c(m,m-d+1)+c(m,m-d+3)+c(m,m-d+5)+\cdots), $$
and this inequality is best possible. (The sum on the right is finite
since $c(m,k)=0$ for $k>m$.)
\end{thm}

\begin{proof}
It is not hard to see that $v(\ca,p)$ will be maximized when the
hyperplanes $H\in\ca$ are as ``generic as possible,'' i.e., the
intersection poset $L_\ca$ is a boolean algebra $B_m$ with all
elements of rank greater than $d$ (including the top element) removed,
and when $p$ is also generic. (Consider the effect of small
perturbations of hyperplanes not in general position.) Assume then
that $L_\ca$ is such a truncated boolean algebra.  Since $L_\ca$
becomes a geometric lattice $\hat{L}_\ca$ when we add a top element,
it follows that the semicone $\sco(\ca)$ satisfies
$L_{\sco(\ca)}\cong \hat{L}_\ca$. Now ordinary truncation $T^i$ and
Dilworth truncation $D_j$ commute (for $i+j<d$, the ambient
dimension). By equation~\eqref{eq:d1bm} we have
$D_1\hat{L}_\ca\cong T^{m-d-1}\Pi_m$. Now \cite[Exam.~3.11.11]{ec1}  
  $$ \chi_{\Pi_m}(q) = (q-1)\cdots (q-m+1) = \sum_{j=0}^{m-1}
    (-1)^j c(m,m-j)q^{m-j-1}. $$
Thus
  $$ \chi_{T^{m-d-1}\Pi_m}(q) = \sum_{j=0}^{d-1} (-1)^j c(m,m-j)
      q^{d-j} + C, $$
for some $C\in\zz$. Since $\chi_{\mathcal{B}}(1)=0$ for any central
arrangement $\mathcal{B}$, we get
  $$ C= -\sum_{j=0}^{d-1} (-1)^j c(m,m-j). $$
Therefore
  \beas v(\ca) & = & (-1)^d\chi_{T^{m-d-1}\Pi_m}(-1)\\ & = &
    \sum_{j=0}^{d-1} c(m,m-j) -(-1)^d\sum_{j=0}^{d-1} (-1)^j
  c(m,m-j)\\ & = & 2(c(m,m-d+1)+c(m,m-d+3)+c(m,m-d+5)+\cdots), 
  \eeas
and the proof follows.
\end{proof}

For fixed $k$, we have that $c(m,m-k)$ is a polynomial in $m$. Hence
for fixed $d$, the bound in Theorem~\ref{thm:max} is a polynomial
$P_d(m)$ in $m$. For instance, 
 \beas P_1(m) & = & 2\\
       P_2(m) & = & m(m-1)\\
       P_3(m) & = & \frac{1}{12}(2m^4-10m^3+9m^2-2m+4)\\
       P_4(m) & = & \frac{1}{24}m(m-1)(m^4-6m^3+11m^2-6m+24)\\
       P_5(m) & = &
       \frac{1}{2880}
   (15m^8-180m^7+830m^6-1848m^5+2735m^4-3300m^3\\
     & & \quad +2180m^2-432m+5760). 
     \eeas

Clearly given $m>d$ we can find a convex $d$-polytopes with $m$
facets, where the affine spans of the facets are as ``generic as
possible,'' as defined at the beginning of the proof of
Theorem~\ref{thm:max}. Thus we obtain the following corollary to
Theorem~\ref{thm:max}.

\begin{coro}
Let $\cp$ be a convex polytope in $\rr^d$ with $m$ facets, and let
$p$ be a point in the interior of $\cp$. Then the number
$\mathrm{ls}(\cp,p)$ of line shellings of $\cp$ whose shelling line
passes through $p$ satisfies 
  $$ \mathrm{ls}(\cp,p)\leq 2(c(m,m-d+1)+c(m,m-d+3)
    +c(m,m-d+5)+\cdots), $$ 
and this inequality is best possible. 
\end{coro}

\section{Further vistas}
We have considered the intersection of a line $L$ through a point $p$
with the hyperplanes of an arrangement $\ca$. We will sketchily
describe an extension. Namely, what if we replace $L$
with an $m$-dimensional plane (or $m$-plane for short) $P$ through $m$ 
points $p_1,\dots,p_m$ not lying on any $H\in\ca$? We will obtain an
induced arrangment
  $$ \ca_P=\{ H\cap P\st H\in\ca\} $$
in the ambient space $P$. Define the generalized valid order
arrangment vo$(\ca;p_1,\dots,p_m)$ to consist of all hyperplanes
passing through $p_1,\dots,p_m$ and every intersection of $m~+~1$
hyperplanes of $\ca$, including ``intersections at $\infty$. The
regions of vo$(\ca;p_1,\dots,p_m)$ correspond to the different
equivalence classes of arrangements $\ca_P$, where $\ca_P$ and $\ca_Q$
are considered equivalent if they correspond to the same
\emph{oriented} matroid. We then have the following analogue of
Theorem~\ref{thm:dt}.

\begin{thm} \label{thm:dt2}
Let $\ca$ be an arrangement in the real vector space $V$, and let
$p_1,\dots,p_m$ be ``sufficiently generic'' points of $V$. Then
$L_{\vo(\ca;p_1,\dots,p_m)}\cong L_{D_m(\ca)}$.
\end{thm}

\begin{figure}
\centering
\centerline{\includegraphics[width=6cm]{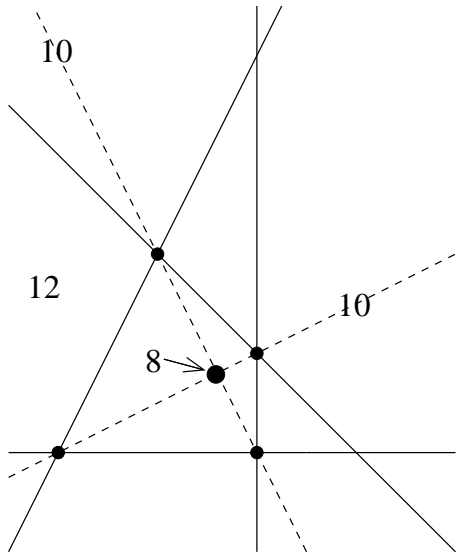}}
\caption{An example of a polyhedral decomposition $\Gamma$ associated
  to valid orderings} 
\label{fig:nong}
\end{figure}

Theorem~\ref{thm:dt} deals with $\vo(\ca,p)$ when $p$ is generic. What
about nongeneric $p$? Define two points $p,q$ not lying on any
hyperplane of $\ca$ to be \emph{equivalent} if there is a canonical
bijection $\varphi\colon\vo(\ca,p)\to\vo(\ca,q)$. By \emph{canonical},
we mean that if $H$ is a hyperplane of $\vo(\ca,p)$ which is the
affine span with $p$ and the intersection $H_1\cap H_2$ of two
hyperplanes in $\ca$ (including an intersection at $\infty$, i.e., $H$
is parallel to $H_1$ and $H_2$), then $\varphi(H)$ is the affine span
of $q$ and $H_1\cap H_2$. The equivalence classes of this equivalence
relation form a polyhedral decomposition of
$\rr^d$. Figure~\ref{fig:nong} shows an example. The arrangement $\ca$
is given by solid lines, and the lines (1-faces) of the polyhedral
decomposition $\Gamma$ by broken lines. Each face $F$ of $\Gamma$ is
marked with the number $v(\ca,p)$ of valid orderings for $p\in F$.

What can be said about the polyhedral complex $\Gamma$? The
2-dimensional case illustrated in Figure~\ref{fig:nong} is somewhat
misleading. Let $\ca$ be an arrangement in $\rr^d$, and let
$p\in\rr^d-\bigcup_{H\in\ca} H$. Suppose that $H_1,\dots,H_4\in\ca$
with $H_1\neq H_2$ and $H_3\neq H_4$. If aff$(p,H_1\cap
H_2)=\aff(p,H_3\cap H_4)$, then the two $(d-2)$-dimensional subspaces
$H_1\cap H_2$ and $H_3\cap H_4$ must both lie on an affine hyperplane
$K$. If $d=2$ then this condition always holds, but for $d>2$ it does
not hold for ``generic'' $\ca$. Thus for generic $\ca$ and $d>4$, the
valid order arrangements $\vo(\ca,p)$ have the same number of
hyperplanes for any $p$. However, they may still differ in how the
hyperplanes intersect. It may be interesting to further investigate
the properties of $\Gamma$.

\end{document}